\documentclass[a4paper,12pt]{amsart}
\usepackage{amsmath,amsthm,amssymb,graphicx}
\newtheorem{thm}{Theorem}[section]
\newtheorem{cor}[thm]{Corollary}
\newtheorem{lem}[thm]{Lemma}

\usepackage[margin=1in]{geometry}
\allowdisplaybreaks[1]

\begin{document}

\title[Bounds for boxicity of some classes of graphs]{Bounds for boxicity of some classes of graphs}%$G$-generalized Join of Graphs.}

\author{T. Kavaskar}
\address{Department of Mathematics, Central University of Tamil Nadu, Thiruvarur-610005, India.}
\email{t\_kavaskar@yahoo.com}


\begin{abstract}
Let $box(G)$ be the boxicity of a graph $G$, $G[H_1,H_2,\ldots, H_n]$ be the $G$-generalized join graph of $n$-pairwise disjoint graphs $H_1,H_2,\ldots, H_n$, $G^d_k$ be a circular clique graph (where $k\geq 2d$) and $\Gamma(R)$ be the zero-divisor graph of a commutative ring $R$.
In this paper, we prove that $\chi(G^d_k)\geq box(G^d_k)$, for all $k$ and $d$ with $k\geq 2d$. This generalizes the results proved in \cite{Aki}. Also we obtain that $box(G[H_1,H_2,\ldots,H_n])\leq \mathop\sum\limits_{i=1}^nbox(H_i)$. As a consequence of this result, we obtain a bound for boxicity of zero-divisor graph of a finite commutative ring with unity. In particular, if $R$ is a finite commutative non-zero reduced ring with unity, then $\chi(\Gamma(R))\leq box(\Gamma(R))\leq 2^{\chi(\Gamma(R))}-2$. where $\chi(\Gamma(R))$ is  the chromatic number of $\Gamma(R)$. 
Moreover, we show that if $N= \prod\limits_{i=1}^{a}p_i^{2n_i} \prod\limits_{j=1}^{b}q_j^{2m_j+1}$ is a composite number, where $p_i$'s and $q_j$'s are distinct  prime numbers, 
then $box(\Gamma(\mathbb{Z}_N))\leq \big(\mathop\prod\limits_{i=1}^{a}(2n_i+1)\mathop\prod\limits_{j=1}^{b}(2m_j+2)\big)-\big(\mathop\prod\limits_{i=1}^{a}(n_i+1)\mathop\prod\limits_{j=1}^{b}(m_j+1)\big)-1$, 
where $\mathbb{Z}_N$ is the ring of integers modulo $N$. 
Further, we prove that, $box(\Gamma(\mathbb{Z}_N))=1$ if and only if  either $N=p^n$ for some prime number $p$ and some positive integer $n\geq 2$ or $N=2p$ for some odd prime number $p$. 

\end{abstract}

\maketitle

\section{Introduction}
In this paper, all graphs considered are simple and undirected. Given a non-empty set $\Upsilon$ and a collection $\mathbf{D}$ of subsets of $\Upsilon$, the \textit{intersection graph} of $\mathbf{D}$ is defined as the graph with vertex set $\mathbf{D}$ such that two elements of $\mathbf{D}$ are adjacent if and only if their intersection is non-empty. An $\ell$-box is the Cartesian product $[x_1,\ y_1]\times \ldots \times [x_{\ell},\ y_{\ell}]$ of $\ell$ closed and bounded intervals of the real line. The \textit{boxicity} of a graph $G=\left(V(G),E(G)\right)$, denoted by $box(G)$, is the least  positive integer $\ell$ such that $G$ is isomorphic to the intersection graph of a family of $\ell$-boxes in Euclidean $\ell$-space. 
From a computational point of view, the boxicity of a graph is an NP-hard parameter, which was proved by Yannakakis \cite{Mih}. However, various bounds have been found in terms of the maximum degree, acylic chromatic number, weak $r$-coloring number, genus, etc. See for instance \cite{Adi}, \cite{Cha1}, \cite{Cha2}, \cite{Cha3}-\cite{Esp2}. 
%For computing the boxicity of a graph is NP-hard proved by Mihalis Yannakakis in \cite{Mih}. 
Boxiity is used as a measure of the complexity of ecological \cite{Rob1} and social \cite{Free} networks, and has applications in fleet maintenance \cite{Ops}. 
%Varies bounds for boxicity of a graph in terms of its maximum degree, acylic chromatic number, weak $r$-coloring number, genus etc., have been studied by several authors, see (\cite{Adi}, \cite{Cha1}, \cite{Cha2}, \cite{Cha3}-\cite{Esp2}). %\cite{Cha4} \cite{Esp1}, \cite{Esp2}).  
In \cite{Cha2}, S. Chandran et al., proved that $box(G)\leq 2\Delta^2(G)$ for any graph $G$, where $\Delta(G)$ is the maximum degree of $G$. This bound was improved by L. Esperet in \cite{Esp3} to $box(G)\leq \frac{\left\lfloor  \Delta(G)^2 \right\rfloor}{2}+2$. 

%they conjectured that $box(G)$ is $O(\Delta(G))$ for any graph $G$. However, this conjecture was disproved by Adiga, Bhowhmick, and Chandran \cite{Adi}.

For a given finite class of graphs $H_i=(V,E_i)$, $i=1,2,\ldots, k$ with same vertex set $V$, 
the \textit{intersection} $H_1\cap H_2\cap \ldots \cap H_k$ of the graphs $H_1,H_2,\ldots, H_k$ is the graph with vertex set $V$ and edge set $E_1\cap E_2\cap\ldots \cap E_k$. A graph $G$ is said to be an \textit{interval graph} if there is a function $f$ from the vertex set of $G$ into a set of closed and bounded intervals in the real line such that $uv\in E(G)$ if and only if $f(u)\cap f(v)\neq \emptyset$. Such a function $f$ is called an \textit{interval representation} of $G$. Note that $box(G)=1$ if and only if $G$ is an interval graph.

Let us recall that a result proved by Roberts in \cite{Rob}. 

\begin{thm}[\cite{Rob}]\label{1a}
Let $G$ be a graph. Then $box(G)\leq \ell$ if and only if there exist
$\ell$ interval graphs $I_1,\ldots, I_{\ell}$ with vertex set $V(G)$ such that $E(G)=E(I_1)\cap\ldots \cap E(I_{\ell})$ holds.
\end{thm}

%\noindent We now recall an equivalent characterisation of boxicity of a graph given by Roberts in \cite{Rob}.

%Using the this result we prove the following result. 

For a given $S\subseteq V(G)$, the \textit{subgraph of $G$ induced} by $S$ is denoted by $\langle S\rangle$. % the \textit{induced subgraph} of $G$. 
A \textit{clique} in $G$ is a complete subgraph of $G$. The \textit{clique number} of a graph $G$, denoted by $\omega(G)$, is the largest size of a clique in $G$. A subset $U\subseteq V(G)$ is said to be \textit{independent} if no two vertices in $U$ are adjacent in $G$.

Let $k$ be a positive integer. A \textit{proper $k$-coloring} of a graph $G$ is a function from $V(G)$ into the set of $k$ colors such that no two adjacent vertices receive the same color. 
The \textit{chromatic number} of a graph $G$, denoted by $\chi(G)$, is the least positive integer $k$ such that there exists a proper $k$-coloring of $G$.%, that is, the vertex set of $G$ can be partition into $k$-independent subset of $G$. 

Let $k$ and $d$ be two positive integers with $k\geq 2d$. The well-known \textit{circular clique} $G^d_k$ has the vertex set $V(G^d_k)= \{a_j\ |\ 0\leq j\leq k-1\}$
and the edge set $E(G^d_k)=\{a_ia_j | \ d\leq |i-j|\leq k-d\}$. The circular clique graph $G^d_k$ is a special class of circulant graphs which play an important role in the study of circular chromatic number. This can be seen in  \cite{Vince, Xu}. Note that, the circular clique graph is vertex-transitive. 

In \cite{Bho}, Bhowmick et al., proved that $box(G)\leq \chi(G)$ holds for any graph $G$ with no asteroidal triples. This result is not true in general. For example, the graph $G$ obtained from a complete bipartite graph with $2n$ (where $n\geq 5$) vertices by removing a perfect matching has $box(G)=\left\lceil \frac{n}{4}\right\rceil > 2 = \chi(H)$ (see \cite{Cha}). %for detail).
S. Charndran et al., in \cite{Cha0} observed that almost all graphs have boxicity more than their chromatic number by using probabilistic method. Interestingly, the family of graphs with boxicity at most their chromatic number is also not narrow. 
%which is based on the probabilistic method. But the family of graphs with boxicity at most their chromatic number is not narrow. 
So it is interesting to characterize graphs $G$ for which $\chi(G)\geq box(G)$ holds. 
In this direction, Akira Kamibeppu \cite{Aki} gave a sufficient condition for $box(G) \leq \chi(G)$. Moreover they showed that $box(G) \leq \chi(G)$ hold for a special class of circular clique $G^d_k$ with an asteroidal triple.

%holds for a graph $G$ in a special class of circular clique $G^d_k$ with an asteroidal triple as follows.

\begin{thm}[\cite{Aki}]\label{1.1a}
For a circulant graph $G^d_{nd}$ with $n \geq 2$ and $d\geq 1$, the inequality
$\chi(G^d_{nd})\geq box(G^d_{nd})$ holds.
\end{thm}

\begin{thm}[\cite{Aki}]\label{1.2b}
For a circulant graph $G^d_{nd+r}$ with $n \geq 2$, $d\geq 2$ and $1\leq r< d$, if $n\geq d-r-1$, the inequality $\chi(G^d_{nd+r})\geq box(G^d_{nd+r})$ holds.
\end{thm}

We recall the definition of  $G$-generalized join of graphs (\cite{sch}). 
Let $G$ be a graph with $V(G)=\{u_1,u_2,\ldots, u_n\}$ and $H_1,H_2,\ldots, H_n$ be pairwise disjoint graphs. The \textit{$G$-generalized join} $G[H_1,H_2,\ldots, H_n]$ of $H_1,H_2,\ldots, H_n$ is the graph obtained by replacing each vertex $u_i$ of $G$ by $H_i$ and joining each vertex of $H_i$ to each vertex of $H_j$ if $u_i$ is adjacent to $u_j$ in $G$. Note that, if $G=K_2$, then the $G$-generalized join $K_2[H_1,H_2]$ of $H_1$ and $H_2$ is the join $H_1\vee H_2$ of $H_1$ and $H_2$.

In \cite{And2}, an equivalence relation $\mathtt{\sim}_G$ for any graph $G$ was introduced as follows: For any $x,y \in V(G)$, $x \mathtt{\sim}_G y$ if and only if $N(x) = N(y)$. Let $[x]$ denote the equivalence class which contains $x$, and $S=\{[x_1],[x_2],\ldots, [x_k]\}$ be the set of all equivalence classes of $\mathtt{\sim}_G$. Clearly, the set of vertices in a class $[x]$ is an independent subset of $G$. The \textit{reduced graph} of $G$ is a simple graph $G_r$ with vertex set $V(G_r)=S$ and two distinct vertices $[x]$ and $[y]$ are adjacent in $G_r$ if and only if $x$ and $y$ are adjacent in $G$. Note that, any graph $G$ is the $G_r$-generalized join of $\left\langle [x_1]\right\rangle,\ldots, \left\langle [x_k]\right\rangle$, that is, $G=G_r\big[\left\langle [x_1]\right\rangle,\ldots, \left\langle [x_k]\right\rangle\big]$.

%The graph $G_r$ is called the \textit{reduced graph} of G. 

%Let $G$ be a graph. Define a relation $\mathtt{\sim}_G$ on $V(G)$ by $u\mathtt{\sim}_G v$ if and only if $N(u)=N(v)$. It is proved in \cite{} $\mathtt{\sim}_G$ is an equivalence relation. For each vertex $v\in V(G)$, let $v^*$ be the set of vertices of $G$ that are equivalent to $v$ under $\mathtt{\sim}_G$. Let $\{v^*_1,v^*_2,\ldots, v^*_k\}$ be the partition of $V(G)$ induced by $\mathtt{\sim}_G$, where each $v_i$ is a representative of the set $v^*_i$. The twin graph of $G$, denoted by $G^*$, is the graph with vertex set $V(G^*) :=\{v^*_1,v^*_2,\ldots, v^*_k\}$, where $v^*_iv^*_j\in E(G^*)$ if and only if $v_iv_j\in E(G)$.

Throughout this paper, $R$ denotes a non-zero finite commutative ring with unity 1. A non-zero element $x$ of $R$ is said to be a \textit{zero-divisor} if there exists a non-zero element $y$ of $R$ such that $xy$=0. An element $u$ of $R$ is \textit{unit} in $R$ if there exists $v$ in $R$ such that $uv$=1. % A ring $R$ is said to be \textit{local} if it has unique maximal ideal. 
The \textit{nilradical} of a ring $R$ is the set $J=\{x\in R : x^t=0,$ for some positive integer $t\}$. %The \textit{index of nilpotency} of $J$ is the least positive integer $m$ for which $J^m=\{0\}$, where $J^m=JJ\ldots J$ ($m$-times). 
A ring $R$ is said to be \textit{reduced} if $J=\{0\}$. 
%Let $R$ be a commutative ring with unity $1\neq 0$. A non-zero element $x\in R$ is said to be a \textit{zero-divisor} if there is a non-zero element $y\in R$ such that $xy=0$. 
 For any ring $R$, in \cite{Beck}, Beck associated a simple graph with $R$ whose vertices are the elements of $R$ and any two distinct vertices $x$ and $y$ are adjacent if and only if $xy$=0 in $R$. %We use $R$ to denote both the ring as well its associated graph. 
It can be observed that for the graph associated with the ring, the vertex $0$ is adjacent with every other vertex. D.F. Anderson et. al. in \cite{And3} slightly modified the definition of the graph associated with a ring by considering the nonzero zero divisors as the vertices and any two distinct vertices $x$ and $y$ are adjacent if and only if $xy=0$ in $R$. They called this as the \textit{zero-divisor graph} of the ring and it is denoted by $\Gamma(R)$. The zero-divisor graph have been studied extensively by several authors. See for instance \cite{Akb}-\cite{And3}, \cite{Sri}, \cite{Hal}, \cite{Kavas}.

%The \textit{zero-divisor graph $\Gamma (R)$} of $R$ is the simple graph with vertex set consisting of all zero-divisors of $R$ such that two distinct vertices $x$ and $y$ are adjacent if and only if $xy=0$. The concept of zero-divisor graph was introduced by Beck in \cite{Beck}. Later, the zero-divisor graph was studied extensively by several authors, see \cite{Akb}-\cite{And3},\cite{Hal},\cite{Sri}.   

For $x,y\in R$, define $x\mathtt{\sim}_R y$ if and only if $Ann(x)=Ann(y)$. It is proved in \cite{mul} that the relation $\mathtt{\sim}_R$ is an equivalence relation on $R$. For $x\in R$, let $C_x=\{r\in R\ |\ x\mathtt{\sim}_Rr \}$ be the equivalence class of $x$. Let $R_E=\{C_{x_1},C_{x_2},\ldots, C_{x_k}\}$ be the set of equivalence classes of the relation $\mathtt{\sim}_R$ other than $C_0$ and $C_1$. The \textit{compressed zero-divisor graph} $\Gamma_E(R)$ (defined in \cite{spi}) is the simple graph with vertex set $R_E$%\backslash \{C_0,C_1\}$ 
and two distinct vertices $C_x$ and $C_y$ are adjacent if and only if $xy$=0. The following result can be found in \cite{sel}.
\begin{thm}[\cite{sel}]\label{A}
If $R$ is a finite commutative ring with unity, then\\
(i)\ \ \ $\Gamma(R)\cong \Gamma_E(R)[\left\langle C_{x_1}\right\rangle,\left\langle C_{x_2}\right\rangle,\ldots, \left\langle C_{x_k}\right\rangle]$, \\
 (ii)\ \ $\left\langle C_{x_i}\right\rangle$ is complete if and only if $x_i^2=0$, and \\
(iii)\ $\left\langle C_{x_i}\right\rangle$ is totally
disconnected (that is, complement of complete graph) if and only if $x_i^2\neq 0$.
\end{thm}

This paper is organized as follows. 
In the second section, we prove that $\chi(G^d_k)\geq box(G^d_k)$, for any $k\geq 2d$ and as a consequence we have Theorems \ref{1.1a} and  \ref{1.2b}. In fact, this was a question raised in \cite{Aki}. 

In the third section, we obtain 
$M \leq box\left(G[H_1,H_2,\ldots, H_n]\right)\leq \mathop\sum\limits_{i=1}^nbox(H_i),
$ if at least one of $H_i$'s is not complete and $M=\max \Big\{
\mathop\sum\limits_{j=1}^k box(H_{i_j})\ | \{u_{i_1},\ldots, u_{i_k}\} $ is a clique in $G$ and each $H_{i_j}$ is not a complete graph$\Big\}$. 
%$\Omega$, where  \begin{eqnarray} \Omega&=& \left\{ \mathop\sum\limits_{j=1}^k box(H_{i_j})\ | \{u_{i_1},\ldots, u_{i_k}\}  \mbox{ is a clique in } G \mbox{ and each } H_{i_j} \mbox{ is not }\right.\\ &&\qquad \left. \mbox{ \ \ \ \ \ \ \ \ \ \ \ \ \ \ \ \ \ \ \ \ a complete graph } \right\}. \nonumber \end{eqnarray}
As a consequence of this result, we prove that if $G_r$ is the reduced graph of $G$, then $box(G)\leq |V(G_r)|$.

In the fourth section, we obtain a bound for the boxicity of zero-divisor graph of a ring. In particular, we prove that $\chi(\Gamma(R))\leq box(\Gamma(R))\leq 2^{\chi(\Gamma(R))}-2$ if $R$ is a non-zero reduced ring.

In the last section, for a composite number $N$, we prove that if $N= \prod\limits_{i=1}^{a}p_i^{2n_i} \prod\limits_{j=1}^{b}q_j^{2m_j+1}$, where $p_i$'s and $q_j$'s are distinct  prime numbers,   % and $a,b$ are non-negative integers, 
then $\omega(\Gamma_E(\mathbb{Z}_N))=\mathop\prod\limits_{i=1}^{a}(n_i+1)\mathop\prod\limits_{j=1}^{b}(m_j+1)+b-1=\chi(\Gamma_E(\mathbb{Z}_N))$, where $\mathbb{Z}_n$ is a ring of integer modulo $N$ and 
$$box(\Gamma(\mathbb{Z}_N))\leq \left(\mathop\prod\limits_{i=1}^{a}(2n_i+1)\mathop\prod\limits_{j=1}^{b}(2m_j+2)\right)-\left(\mathop\prod\limits_{i=1}^{a}(n_i +1)\mathop\prod\limits_{j=1}^{b}(m_j+1)\right)-1.$$
Finally, we have shown that $box(\Gamma(\mathbb{Z}_N))=1$ if and only if either $N=p^n$ for some prime number $p$ and some positive integer $n\geq 2$ or $N=2p$  for some odd prime number $p$. 
\section{Boxicity of Circular Clique $G^d_k$}
%We recall the definition of join of two graphs. Let $G$ and $H$ be two vertex disjoint graphs. The join of $G$ and $H$ is the graph $G \vee H$ with vertex set $V(G\vee H)=V(G)\cup V(H)$ and edge set $E(G\vee H)=E(G)\cup E(H)\cup \{xy \ : \ x\in V(G), y\in V(H)\}$. Denote $\langle T\rangle$ the induced subgraph of $G$ of a subset $T$ of the vertex set of $G$. 
We recall the following result from \cite{Vince, Xu}.
\begin{thm}[\cite{Vince, Xu}]\label{1.1c}
If $k$ and $d$ are any two positive integers with $k\geq 2d$, then $\chi(G^d_k)=\left\lceil \frac{k}{d}\right\rceil$.
\end{thm}

Using Theorems \ref{1a} and \ref{1.1c}, we prove Theorem \ref{TK1}.
\begin{thm}\label{TK1}
If $k$ and $d$ are any two positive integers with $k\geq 2d$, then $\chi(G^d_k)\geq box(G^d_k)$.
\end{thm}
\begin{proof}

Let $k=md+b$, where $0\leq b\leq d-1$.

\noindent \textbf{Case 1.} $m\geq 3$.

Let $r=d$ or $b$ and let $W=\{a_j\ |\ 0\leq j\leq r-1\}$. %,a_1,\ldots, a_{r-1}\}$. 
Then $W$ is an independent subset of $G^d_k$. We now claim the following for $W$.\\
\textbf{Claim 1.} There is an interval graph $I_W$ such that $G^d_k$ is a spanning subgraph of $I_W$ and $W$ is an independent subset of $I_W$. \\
\textit{Proof of claim 1.}

Let $H$ be the complete graph with the vertex set $$V(H)=\{a_j\ |\ r\leq j\leq d-1\}\cup \{a_{d+j}\ |\ r\leq j\leq k-1-d\}$$ and $S$ be the complete graph with the vertex set $V(S)=\{a_{d+j}\ |\ 0\leq j\leq r-1\}$. %a_d,\ldots, a_{d+r-1}\}$. 
Define $I_W$ to be the graph with the vertex set $V(I_W)=V(G^d_k)$ and the edge set  

$E(I_W)=E(G^d_k)\cup E(H)\cup E(S)\cup \{a_ia_{d+j}\ |\ 0\leq i\leq r-1, r\leq j\leq k-d-1\}$

\ \ \ \  \ \ \ \ \ \ \ \ \ \ \  $\cup\{a_ia_{d+j}|\ r\leq i\leq d-1$, or $d+r\leq i\leq k-1 \text{ and } 0\leq j\leq r-1\}.$ \\% , that is, $$E(I_W)=E(G^d_k)\cup E(\langle W \vee H\rangle)\cup E(\langle S \vee H\rangle).$$ 
Then clearly $G^d_k$ is a spanning subgraph of $I_W$, $W$ is an independent subset of $I_W$ and there is no edge between $W$ and $\{a_j: r\leq j\leq d-1\}$ in $I_W$.

We shall now show that $I_W$ is our required interval graph. The following is an interval representation $f$ of $I_W$.\\
\noindent (i)\ For $0\leq i\leq r-1$, 
$$f(a_i)=\{d-i\},$$ and $$f(a_{d+i})=[d-i,\ d+1],$$
%assign the single interval $\{d-i\}$ to $a_i$ and assign the interval $[d-i,d+1]$ to $a_{d+i}$. \\
\noindent (ii)\ For $r\leq i\leq d-1$, $$f(a_i)=\{d+1\},$$ 
(iii) For $d+r\leq i\leq k-1$,
$$f(a_i)=[1,\ d+1].$$
% assign the interval $[1,d+1]$ to $a_i$. 

%We can easily observe the following notes.

%\noindent \textit{Note 1.} 
Note that for $1\leq i\leq r-1$, the number $d-i$ does not belong to the interval $[d-j,d+1]$, for $\ 0\leq j\leq i-1$. 

Let $x,y\in V(I_W)$, say $x=a_i$ and $y=a_j$ where $i<j$. First, if $xy\notin E(I_W)$, then either %$x,y\in W\cup S$. More precisely, 
$x,y\in W$, or $x\in W$ and $y\in \{a_r,a_{r+1},\ldots, a_{d-1}\}$ or $x\in W$ and $y\in S$, where $j\in \{d+1,d+2,\ldots, d+i-1\}$ (as $m\geq 3$). We can observe in all the case that $f(x)\cap f(y)=\emptyset$.
%\noindent \textit{Note 3.} 
Next, if $xy\in E(I_W)$, then either $x,y\in S\cup V(H)$ or $x\in W$ and %for some $0\leq i\leq r-1$ (that is, $x\in W$), then 
$y\in S\cup V(H)$. %where $j\in \{d+i,\ldots, k-1\}$ %(that is, $y\in V(H)\cup S$) 
In both the cases $f(x)\cap f(y)\neq \emptyset$. %More precisely, $x,y\in W$, or if $x=a_i$, for $1\leq i\leq r-1$, then $y=a_{d+j}$, for $0\leq j\leq i-1$.  
\noindent Thus $I_W$ is an interval graph with the desired property. %Hence the Claim 1.%\qed

%By the division algorithm, $k=\lfloor k/d \rfloor d+r$, where $0\leq r< d$. 

\noindent \textbf{Subcase (i)} $b\geq 1$. %Let $n=\left\lceil \frac{k}{d} \right\rceil$. 

By Theorem \ref{1.1c}, $\chi(G^d_k)= m+1$. Let $V(G^d_k)=\mathop\bigcup\limits_{i=0}^m W_i$, where $W_i=\{a_{id+j}\ |\ 0\leq j\leq d-1\}$, %, a_{(i-1)d+1},\ldots, a_{id-1}\}$, 
for $0\leq i\leq m-1$ and $W_m=\{a_{md+j}\ |\ 0\leq j\leq b-1\}$. %,a_{md+1},\ldots,a_{md+r-1}\}$. 
Clearly, $W_i\cap W_j=\emptyset$, for $0\leq i, j\leq m$ with $i\neq j$ and $W_i$ is an independent subset of $G^d_k$, for $0\leq i\leq m$. Since $G_k^d$ is a vertex-transitive graph, for each $0\leq i\leq m$, there is an interval graph $I_{W_i}$ for $W_i$ such that $G^d_k$ is a spanning subgraph of $I_{W_i}$ and $W_i$ is an independent subset of $I_{W_i}$, by Claim 1. 

By Theorem \ref{1a}, it is enough to prove $E(G^d_k)=\mathop\bigcap\limits_{i=0}^{m} E(I_{W_i})$. Let $x, y\in V(G^d_k)$. If they are adjacent in $G^d_k$, then they are adjacent in each $I_{W_i}$, for $0\leq i\leq m$, because $G^d_k$ is a spanning subgraph of $I_{W_i}$. It remains to show that if $x$ and $y$ are not adjacent in $G^d_k$, then they are not adjacent in at least one of  $I_{W_i}$'$s$. Suppose if $x,y\in W_i$, for some $i$, then they are not adjacent in $I_{W_i}$. Therefore, by the definition of $G^d_k$, $x\in W_i$ and $y\in W_{i+1}$, for some $i$ (where the addition in subscript is taken modulo over $m$), then they are not adjacent in $I_{W_i}$ or in $I_{W_{i+1}}$. %Thus $G^d_k=\bigcap_{i=1}^{\lceil k/d \rceil} G_i$ 
Hence $\chi(G^d_k)\geq box(G^d_k)$.%, by Theorem \ref{1}.

\noindent \textbf{Subcase (ii)} $b=0$.

By Theorem \ref{1.1c}, $\chi(G^d_k)= m$. Let $V(G^d_k)=\mathop\bigcup\limits_{i=0}^{m-1} W_i$, where $W_i=\{a_{id+j}\ |\ 0\leq j\leq d-1\}$, for $0\leq i\leq m-1$. By a similar strategy as in Subcase (i), we can prove that $\chi(G^d_k)\geq box(G^d_k)$.

\noindent \textbf{Case 2.} $m=2$.

Then $k=2d+b$, where $0\leq b\leq d-1$. Suppose if $b=0$, then $G^d_{2d}$ is an $1$-regular graph (that is, the degree of every vertex of $G^d_{2d}$ is one) and hence it is an interval graph. Thus, $box(G^d_{2d})=1<2=\chi(G^d_{2d})$. So, let us assume that $b\geq 1$. 

First if $U=\{a_j\ |\ 0\leq j\leq b-1\}$, %a_0,a_1,\ldots, a_{b-1}\}$, 
then by the similar way as given in Claim 1, we can construct an interval graph $I_U$ for $U$ such that $G^d_k$ is a spanning subgraph of $I_U$ and $U$ is an independent subset of $I_U$. Next if $W=\{a_j\ |\ 0\leq j\leq d-1\}$, then we claim the following for $W$.

\noindent \textbf{Claim 2.} There is an interval graph $I_W$ such that $G^d_k$ is a spanning subgraph of $I_W$ and $W$ is an independent subset of $I_W$. \\
\textit{Proof of claim 2.} \\
Since $2\leq b+1\leq d$, there exist non-negative integers $c,e$ such that $d=c(b+1)+e$ where $0\leq e\leq b$. Let $S=\{a_{d+j}\ |\ 0\leq j\leq d-1\}$. %\{a_d,a_{d+1},\ldots, a_{2d-1}\}$. 
We write $S=\mathop\bigcup\limits_{i=0}^{c}S_i$, where $S_i=\{a_{d+i(b+1)+j}\ |\ 0\leq j\leq b \}$ for $0\leq i\leq c-1$ and $S_c=\{a_{d+c(b+1)+j}\ |\ 0\leq j\leq e-1\}$.

Let $H$ be the complete graph with vertex set $V(H)=\{a_{2d+j}\ |\ 0\leq j\leq b-1\}$ %\{a_{2d},\ldots,a_{k-1}\}$ 
and $S'$ be a graph with vertex set $V(S')=S$ %$\{a_{d+j}\ |\ 0\leq j\leq d-1\}$ %$\{a_d,\ldots, a_{2d-1}\}$ 
and edge set $E(S')=\mathop\bigcup\limits_{i=0}^{c}E_i$, where 
\begin{eqnarray*}
E_i
& = & \{a_{d+i(b+1)+j}\ a_{d+i(b+1)+\ell}\ |\ 0\leq j<\ell \leq b\} \\
&  & \cup\{a_{d+(i+1)(b+1)+j}\ a_{d+i(b+1)+\ell}\ |\ 0\leq j\leq b,\ j\leq \ell\leq b\},
\end{eqnarray*}
%$$E_i=\{a_{d+i(b+1)+j}\ a_{d+i(b+1)+\ell}\ |\ 0\leq j<\ell \leq b\}\cup \{a_{d+(i+1)(b+1)+j}\ a_{d+i(b+1)+\ell}\ |\ 0\leq j\leq b,\ j\leq \ell\leq b\},$$ 
for $0\leq i\leq c-2$,
\begin{eqnarray*}
E_{c-1}
& = & \{a_{d+(c-1)(b+1)+j}\ a_{d+(c-1)(b+1)+\ell}\ |\ 0\leq j<\ell\leq b\} \\
&  & \cup\{a_{d+c(b+1)+j}\ a_{d+(c-1)(b+1)+\ell}\ |\ 0\leq j\leq e-1,\ j\leq \ell\leq b\}.
\end{eqnarray*}
and $E_c=\{a_{d+c(b+1)+j}\ a_{d+c(b+1)+\ell}\ |\ 0\leq j< \ell \leq e-1\}.$

Define $I_W$ be the graph with $V(I_W)=V(G^d_k)$ and $$E(I_W)= E(G^d_k)\cup E(H)\cup E(S')\cup \{a_ia_j|\ 0\leq i\leq 2d-1, 2d\leq j\leq k-1\},$$ that is, 
$E(I_W)=E(G^d_k)\cup E(\langle W \vee H\rangle)\cup E(\langle S' \vee H\rangle).$ Then clearly $G^d_k$ is a spanning subgraph of $I_W$ and $W$ is an independent subset of $I_W$.

Now we prove that $I_W$ is an interval graph. The following is an interval representation $f$ of $I_W$. 

\noindent \ For $0\leq i\leq d-1$ (that is, $a_i\in W$),
\begin{eqnarray}
f(a_i)=\{\alpha_i\}, \text{where } \alpha_i\in (i-1,i),
\end{eqnarray}
%$$f(a_i)=\{\alpha_i\},$$ 
%assign the single interval $\{x_i\}$ to $a_i$, where $x_i\in (i-1,i)$;
\noindent \ For $0\leq j\leq b$ (that is, $a_{d+j}\in S_0$),
\begin{eqnarray}
f(a_{d+j})=[-1,\ j],
\end{eqnarray}
%assign the interval $[-1,\ j]$ to $a_{d+j}$; \\
\noindent for $1\leq i\leq c-1$ and $0\leq j\leq b$ (that is, $a_{d+i(b+1)+j}\in S_i$)
\begin{eqnarray}
f(a_{d+i(b+1)+j})=[(i-1)(b+1)+j,\ i(b+1)+j],
\end{eqnarray}
%assign the interval $[(i-1)(b+1)+j,\ i(b+1)+j]$ to $a_{d+i(b+1)+j}$, where $0\leq j\leq b$ \\
and for $0\leq j\leq e-1$ (that is, $a_{d+c(b+1)+j}\in S_c$),
\begin{eqnarray}
f(a_{d+c(b+1)+j})=[(c-1)(b+1)+j,\ c(b+1)+j],
\end{eqnarray}
%assign the interval $[(c-1)(b+1)+j,\ c(b+1)+j]$ to $a_{d+c(b+1)+j}$, where $0\leq j\leq e-1$. 
\noindent \ For $2d\leq i\leq k-1$, (that is, $a_i\in V(H)$)
\begin{eqnarray}
f(a_i)=[-1,\ d].  %assign the interval $[-1,\ c(b+1)+e]$ to $a_i$. 
\end{eqnarray}
Clearly, $f(x)\subseteq [-1,d]$, for every $x\in V(I_W)$.
%\textbf{Note 1.} $-1\in \mathop\bigcap\limits_{x\in S_0\cup V(H)}f(x)$, $i(b+1)\in \mathop\bigcap\limits_{S_i\cup V(H)}f(x)$, for $1\leq i\leq c$  and also note that if $x=a_j\in W$, then $\{\alpha_j\}=\Big(\mathop\bigcap\limits_{y\in V(H)}f(y)\Big)\cap f(x)$. 

Let $x,y\in V(I_W)$. First if $xy\notin E(I_W)$, then $x,y\in W\cup S$. More precisely, one of the following will happen. Let us establish that $f(x)\cap f(y)=\emptyset$ in all the possibilities. 

%\item if $x=a_i\in W$, for some $1\leq i\leq d-1$, then $y=a_{d+j}$, %\in \mathop\bigcup\limits_{j=0}^{i-1}S_j$, 
%for some $0\leq j\leq i-1$ and hence $f(x)=\{\alpha_i\}$ and $f(y)=[-1,j]\subseteq [-1,i-1]$ (by equations (2) and (3)). So, $f(x)\cap f(y)=\emptyset$, as $\alpha_i\in (i-1,i)$.
%\item if $x=a_i\in W$, for some $1\leq i\leq d-b-2$, then $y=a_{d+b+i+j}$, for some $j\geq 1$ with $b+i+j\leq d-1$ (this is because $m=2$). Hence $f(x)=\alpha_i$ and $f(y)\subseteq [i,d]$ (by equations (2), (4) and (5)), so $f(x)\cap f(y)=\emptyset$, as $\alpha_i\in (i-1,i)$.  

\begin{itemize}
\item $x,y\in W$. Then there exist $0\leq i<j\leq d-1$ such that $x=a_i$ and $y=a_j$. Hence $f(x)\cap f(y)=\{\alpha_i\}\cap \{\alpha_j\}=\emptyset$ (by equations (2)).  
\item $x\in W$ and $y\in S$ or $y\in W$ and $x\in S$. Without loss of generality, let $x\in W$ and $y\in S$. Then $x=a_{i}$ for some $i\in \{0,1,\ldots, d-1\}$ and $y=a_{d+j}$, for some $j\in \{0,1,\ldots, i-1\}$ or $j\in \{b+i+1,\ldots, d-1\}$.
%$j\in \{0,1,\ldots, i-1\}$ $x=a_i\in W$, for some $0\leq i\leq d-1$, then either $y=a_{d+j}$, %\in \mathop\bigcup\limits_{j=0}^{i-1}S_j$, for some $0\leq j\leq i-1$ or $y=a_{d+j}$ for some $b+i+1\leq j\leq d-1$. 
Clearly, $f(x)=\{\alpha_i\}\subseteq (i-1, i)$ (by equation (2)). 
When $0\leq j \leq i-1$, $f(y)=f(a_{d+j})\subseteq [-1,j]\subseteq [-1,i-1]$ (by equations (2), (3) and (4)) % and the latter case and hence $f(x)=\{\alpha_i\}$ and $f(y)=[-1,j]\subseteq [-1,i-1]$ (by equations (2) and (3)). So, $f(x)\cap f(y)=\emptyset$, as $\alpha_i\in (i-1,i)$.
%\item if $x=a_i\in W$, for some $1\leq i\leq d-b-2$, then $y=a_{d+b+i+j}$, for some $j\geq 1$ with $b+i+j\leq d-1$ (this is because $m=2$). Hence $f(x)=\alpha_i$ and $f(y)\subseteq [i,d]$ (by equations (2), (4) and (5)), and 
and when $b+i+1\leq j\leq d-1$, $f(y)=f(a_{d+j})=[j-b-1,j]\subseteq [i,d-1]$ (by equations (3) and (4)). In both the cases $f(x)\cap f(y)=\emptyset$. % as $\alpha_i\in (i-1,i)$.  
\item $x,y\in S$. Then $x\in S_i$ and $y\in S_j$ for some $0\leq i,j\leq c$. Since $S_i$'s are complete, $i\neq j$, say $i<j$. %Then $0\leq i\leq c-1$ and $1\leq j\leq c$. 

First, let us consider $j\geq i+2$. %By equations (2) and (3),
%$$ f(x) \subseteq \begin{cases} [-1,b] & \mbox{if $i=0$}\\ [(i-1)(b+1),i(b+1)+b] & \mbox{if $1\leq i\leq c-1$.}\\  \end{cases} $$
Clearly, $f(x)\subseteq [-1,i(b+1)+b]$ and $f(y)\subseteq [(j-1)(b+1),j(b+1)+b]$ and hence $f(x)\cap f(y)=\emptyset$. %as $j\geq i+2$. 
%If $i=0$, then $f(x)\subseteq [-1,b]$ (by equation (2)) and $f(y)\subseteq [(j-1)(b+1),j(b+1)+b]$ (by equations (3) and (4)). So, $f(x)\cap f(y)=\emptyset$ as $j>1$. Next if $i>0$, then $f(x)\subseteq [(i-1)(b+1),i(b+1)+b]$ and $f(y)\subseteq [(j-1)(b+1),j(b+1)+b]$. So, $f(x)\cap f(y)=\emptyset$ as $j>i+1$. 

Next, let us consider $j=i+1$. %When $i=0$, $x=a_{d+\ell}$ for some $0\leq \ell <b$ and $y=a_{d+\ell+p+b+1}$, for some $1\leq p\leq b-\ell$ and therefore $f(x)=[-1,\ell]$ (by equation (2)) and $f(y)=[\ell+p,\ \ell+p+b+1]$ (by equation (3)). So, $f(x)\cap f(y)=\emptyset$, as $p\geq 1$. Next, 
When $0\leq i\leq c-2$,  $x=a_{d+i(b+1)+\ell}$ for some $0\leq \ell < b$ and $y=a_{d+(i+1)(b+1)+\ell+p}$ for some $1\leq p\leq b-\ell$. %Hence by equation (3), 
One can see that $f(x)\subseteq [-1,\ i(b+1)+\ell]$ and $f(y)= [i(b+1)+\ell+p,\ (i+1)(b+1)+\ell+p]$ and thus $f(x)\cap f(y)=\emptyset$, as $p\geq 1$. Similarly when $i=c-1$, we can show that  %$x=a_{d+(c-1)(b+1)+\ell}$ for some $0\leq \ell < e-1$ and $y=a_{d+c(b+1)+\ell+p}$ for some $1\leq p\leq e-1-\ell$ and hence by equations (3) and (4) $f(x)=[(c-2)(b+1)+\ell,\ (c-1)(b+1)+\ell]$ and $f(y)=[(c-1)(b+1)+\ell+p,\ c(b+1)+\ell+p]$ and thus 
$f(x)\cap f(y)=\emptyset$.
\end{itemize}
%In all cases, we have $f(x)\cap f(y)=\emptyset$.

 Next if $xy\in E(I_W)$, then one of the following will happen.
\begin{itemize}
\item $x\in V(H)$ and $y \in V(I_W)$ or vice versa. %for some $0\leq i\leq c$ or $x\in W\cup S$ and $y\in V(H)$, then $f(x)\cap f(y)\neq \emptyset$ (by equations (2) to (5) and Note 1).
\item $x,y \in S_i$, for some $i\in \{0,1,\ldots, c\}$
\item $x\in S_i$ and $y\in S_{i+1}$, for some $i\in \{0,1,\ldots, c\}$ or vice versa.
\item $x\in W$ and $y\in S$ or vice versa.
%if $x\in W$ and $y\in V(S)$, then $x=a_{\ell}$ for some $0\leq \ell \leq d-1$ and hence $y=a_{d+\ell+p}$, for some $0\leq p\leq b$. Therefore, by equation (1), $f(x)=\{\alpha_{\ell}\}\subseteq (\ell-1,\ell)$. If $0\leq \ell\leq b$, then by equations (2) and (3), 
%$$ f(y) = \begin{cases}  [-1,\ell+p] & \mbox{if $0\leq p\leq b-\ell$}\\ [\ell+p-b-1,\ell+p] & \mbox{if $b-\ell+1\leq p\leq b$.}\\ \end{cases} $$
%If $b+1\leq \ell \leq d-1$, then by equations (3) and (4), $f(y)=[\ell+p-b-1,\ell+p]$.  %then $y\in S_i$, for some $0\leq i\leq c$ and hence there exists $0\leq \ell \leq b$ such that $y=a_{d+i(b+1)+\ell}$ which implies that $x=a_$
%Hence, $f(x)\cap f(y)\neq \emptyset$, as $f(x)\subseteq f(y)$.
%\item if $x\in S_i$ and $y\in S_{i+1}$ for some $0\leq i\leq c-1$, then $x=a_{d+i(b+1)+\ell}$ for some $0\leq \ell \leq b$ and hence by equations (2)-(4), %$f(x)=[(i-1)(b+1)+\ell, i(b+1)+\ell]$ and 
%$$
%f(x) = \begin{cases} 
%[-1,\ell] & \mbox{if $i=0$}\\

%[(i-1)(b+1)+\ell, i(b+1)+\ell] & \mbox{if $1\leq i\leq c-1$.}\\

%\end{cases}
%$$
%and 
%$$
%y = \begin{cases} 
%a_{d+(i+1)(b+1)+p} & \mbox{if $0\leq i\leq c-1$, for some $0\leq p\leq \ell$}\\

%a_{d+c(b+1)+p} & \mbox{if $i=c$, for some $0 \leq p\leq \ell \leq e-1$.}\\

%\end{cases}
%$$
%and thus  (by equations (3) to (4))
%$$
%f(y) = \begin{cases} 
%[i(b+1)+p,(i+1)(b+1)+p] & \mbox{if $0\leq i\leq c-1$,}\\

%[(c-1)(b+1)+p,c(b+1)+p] & \mbox{if $i=c$.}\\

%\end{cases}
%$$
%Therefore, $f(x)\cap f(y)\neq \emptyset$.
\end{itemize}	
By using the definition of $f$, without much difficulty, one can see that $f(x)\cap f(y)\neq \emptyset$.
\noindent Thus, $I_W$ is an interval graph. Hence the Claim 2.\\
By Theorem \ref{1.1c}, $\chi(G^d_k)= 3$. Let $V(G^d_k)=\mathop\bigcup\limits_{i=0}^{2} W_i$, where $W_i=\{a_{id+j}\ |\ 0\leq j\leq d-1\}$, for $0\leq i\leq 1$ and $W_2=\{a_{2d+j}\ |\ 0\leq j\leq b-1\}$. 
Here also similar to Case 1, we can prove that $\chi(G^d_k)=3\geq box(G^d_k)$.
\end{proof}
\section{Bounds for Boxicity of $G$-Generalized Join of Graphs}
In this section we first obtain a bounds for boxicity of $G$-generalized join of graphs.
%Using Theorem \ref{1a}, we prove the following result. %First we recall the following result proved by Roberts \cite{Rob}. 
%\begin{thm}[\cite{Rob}]\label{1}
%Let $G$ be a graph. Then $box(G)\leq \ell$ holds if and only if there exist interval graphs $I_1,\ldots, I_{\ell}$ such that $G=I_1\cap\ldots \cap I_{\ell}$.
%\end{thm}
%Using this result we prove following result.
\begin{thm}\label{3}
Let $G$ be a graph with vertex set $V(G)=\{u_1,\ldots, u_n\}$ and $H_1,\ldots, H_n$ be $n$-pairwise disjoint graphs such that at least one of $H_i$'s is not complete. Then 
%\begin{equation}\label{best}
$M\leq box(G[H_1,H_2,\ldots, H_n])\leq \mathop\sum\limits_{i=1}^n box(H_i)$, where
$M=\max \Big\{
\mathop\sum\limits_{j=1}^k box(H_{i_j})\ | \{u_{i_1},\ldots, u_{i_k}\} $ is a clique in $G$ and each $H_{i_j}$ is not a complete graph$\Big\}$.
%\end{equation}
% \begin{eqnarray*} M=\max \left\{ \mathop\sum\limits_{j=1}^k box(H_{i_j})\ | \{u_{i_1},\ldots, u_{i_k}\}  \mbox{ is a clique in } G \mbox{ and each } H_{i_j} \mbox{ is not a complete graph}\right\}.\\ &&\qquad %\left. \mbox{ \ \ \ \ \ \ \ \ \ \ \ \ \ \ \ \ \ \ \ \  graph } \right\}. \nonumber \end{eqnarray*}
  %$\Omega$ is defined in equation (1). 
\end{thm}
\begin{proof}  It is easy to note that 
\begin{enumerate}
\item[(i)] if $G$ and $H$ are not complete graphs, then $box(G\vee H)\geq box(G)+box(H)$; 
\item[(ii)] if $G'$ is an induced subgraph of $G$, then $box(G)\geq box(G')$. 
\end{enumerate}
If $\langle\{u_{i_1},\ldots, u_{i_k}\}\rangle$ is a clique in $G$, $H_{i_j}$ is not a complete graph, for $1\leq j\leq k$, then $G[H_{i_1},H_{i_2},\ldots, H_{i_k}]=H_{i_1}\vee H_{i_2}\vee \ldots \vee H_{i_k}$ is an induced subgraph of $G[H_1,H_2,\ldots, H_n]$. Thus by (i) and (ii), $$box(G[H_1,H_2,\ldots, H_n])\geq box(G[H_{i_1},H_{i_2},\ldots, H_{i_k}])\geq \mathop\sum\limits_{j=1}^kbox(H_{i_j}),$$ by induction on $k$ and hence %$box(G[H_1,H_2,\ldots, H_n])$ is an upper bound of the set $\Omega$ mentioned in (1). Therefore, $box(G[H_{i_1},H_{i_2},\ldots, H_{i_k}])\geq M$. Hence 
the lower bound follows. 
%\geq box(H_{i_1}+\ldots +box(H_{i_k}))

Next, we prove the upper bound.  
For $1\leq i\leq n$, let $box(H_i)=k_i$. Then, by Theorem \ref{1a}, there exist  interval graphs $I_i^1,\ldots, I_i^{k_i}$ with vertex set $V(H_i)$ such that $H_i=\mathop\bigcap\limits_{j=1}^{k_i} I_i^j$. %, where $I_i^j$ is an interval graph super graph of $H_i$. 
We now define an interval graph $J_i^j$ such that $G[H_1,H_2,\ldots, H_n]$ is a spanning subgraph of $J_i^j$, for all $1\leq i\leq n$ and $1\leq j\leq k_i$. 

For $1\leq i\leq n$, denote by $N_G(u_i)$ the set 
of all vertices $u_{i_1},u_{i_2},\ldots u_{i_t}$ of $G$ which are adjacent to $u_i$ in $G$. 
%Let $N_G(u_i)=\{u_{i_1}, u_{i_2},\ldots, u_{i_t}\}$. 
Let $R_i$ be the graph with vertex set $V(G)$ and edge set $$E(R_i)=\{xu_i \ | \ x\in N_G(u_i)\}\cup \{xy \ | \ x\neq u_i \mbox{ and }  y\neq u_i \}.$$ Then $G$ is a spanning subgraph of every $R_i$. For every $1\leq \ell \leq n$ with $\ell \neq i$, let $S_{\ell}$ be the complete graph with vertex set $V(H_{\ell})$. Define $J_i^j=R_i[S_1,\ldots,S_{i-1}, I_i^j, S_{i+1},\ldots, S_n]$, for $1\leq i\leq n$ and $1\leq j\leq k_i$. Clearly $G[H_1,H_2,\ldots, H_n]$ is a spanning subgraph of $J_i^j$. 

\noindent \textbf{Claim 1.} $J_i^j$ is an interval graph, for $1\leq i\leq n$ and $1\leq j\leq k_i$. 

%\begin{proof}
\noindent\textit{Proof of Claim 1.}
Let $c_i^j$ be an interval representation of $I_i^j$. Then there exist real numbers $x$ and $y$ such that $\mathop\bigcup\limits_{v\in V(I_i^j)}c_i^j(v)$ is contained in $[x,y)$. 
%proper subset of the closed and bounded interval $[x,y]$. 
We construct an interval representation $f_i^j$ of $J_i^j$ as follows. 
$$
f_i^j(v) = \begin{cases} 
[x,y] & \mbox{if $v\in S_{i_\gamma}$, $1\leq \gamma \leq t$}\\

[y,z] & \mbox{if $v\in S_{\mu}$, } \mbox{$\mu\in \{1,2,\ldots, n\}\setminus \{i,i_1,\ldots, i_t\}$,}\\
c_i^j(v)& \mbox{if $v\in I_i^j$.}
\end{cases}
$$

First, note that if $vw\notin E(J_i^j)$, then either $v,w\in V(I_i^j)$ or $v\in V(I_i^j)$ and $w\in S_{\mu}$, for some $\mu \in \{1,2,\ldots, n\}\setminus \{i, i_1,\ldots,i_t\}$, say. 
In the first case, $f_i^j(v)\cap f_i^j(w)=c_i^j(v)\cap c_i^j(w)=\emptyset$, and the later case, the interval $f_i^j(v)=c_i^j(v)$ is a subset of $[x,y)$ and $f_i^j(w)=[y,z]$ and hence $f_i^j(v)\cap f_i^j(w)=\emptyset$. 

If $vw\in E(J_i^j)$, then either $v,w\in V(I_i^j)$ or $v\in V(I_i^j)$ and $w\in \mathop\bigcup\limits_{\gamma=1}^t V(S_{i_\gamma})$, or $v, w\not\in V(I_i^j)$. 
In all cases, $f_i^j(v)\cap f_i^j(w)\neq \emptyset$. Hence the claim 1. %\hfill
%\end{proof}

\noindent \textbf{Claim 2.} $E(G[H_1,\ldots,H_n])=\mathop\bigcap\limits_{i=1}^n\mathop\bigcap\limits_{j=1}^{k_i}E(J_i^j)$.

As $G[H_1,\ldots,H_n]$ is a spanning subgraph of $J_i^j$, for all $1\leq i\leq n,\ 1\leq j\leq k_i$,  it is enough to check that if $v$ and $w$ are not adjacent in $G[H_1,\ldots,H_n]$, then they are not adjacent in $J_i^j$ for some $i\in \{1,2,\ldots, n\}$ and $j\in \{1,2,\ldots, k_i\}$. Let $vw\notin E(G[H_1,\ldots,H_n])$. First, if $v,w\in V(H_i)$, for some $i\in \{1,2,\ldots, n\}$, then they are not adjacent in $I_i^j$ for some $j\in \{1,2,\ldots, k_i\}$ and hence they are not adjacent in $J_i^j$. 
%\noindent \textbf{Case 1.} $v,w\in V(H_i)$, for some $i$. 
%If $v$ and $w$ are adjacent in $G[H_1,\ldots,H_n]$, then they are adjacent in $H_i$ and hence they are adjacent in $I_i^j$ for all $1\leq j\leq k_i$. This implies that  $v$ and $w$ are  adjacent in $J_i^j$ for all $1\leq i\leq n$ and $1\leq j\leq k_i$. 
%First, if $v$ and $w$ are not adjacent in $G[H_1,\ldots,H_n]$, then they are not adjacent in $I_i^j$ for some $j$ and hence they are not adjacent in $J_i^j$. 
Next, if $v\in V(H_i)$ and $w\in  V(H_j)$ for $i\neq j$, then $u_i$ and $u_j$ are not adjacent in $G$ and hence $v$ and $w$ are not adjacent in $J_i^j$. Hence the Claim 2.\\
%\noindent \textbf{Case 2.} $v\in V(H_i)$ and $w\in  V(H_j)$ for $i\neq j$. 
%If $v$ and $w$ are adjacent in $G[H_1,\ldots, H_n]$, then $u_i$ and $u_j$ are adjacent in $G$. Hence they are adjacent in $J_i^j$ for all $1\leq i\leq n$ and $1\leq j\leq k_i$. 
%If $v$ and $w$ are not adjacent in $G[H_1,\ldots, H_n]$, then $u_i$ and $u_j$ are not adjacent in $G$ and hence $v$ and $w$ are not adjacent in $J_i^j$. Hence the Claim 2.
%\hfill
%\end{proof} %Hence the Claim 2. 
From Claims 1, 2 and by Theorem \ref{1a}, the result follows.
\end{proof}
The following corollaries are the immediate consequence of Theorem \ref{3}. 
%Now we have the following immediate corollary of Theorem \ref{3}.
%Corollary \ref{4} follows from Theorem \ref{3}.
\begin{cor}
If $G$ is a graph and $G_r$ is the reduced graph of $G$, then $ box(G)\leq |V(G_r)|$.
\end{cor}
%\begin{proof} As $G\cong G_r\big[[x_1],[x_2],\ldots, [x_k]\big]$, for some positive integer $k$ and each $\langle [x_i]\rangle$ has no edge, we have $box(\langle [x_i]\rangle)=1$ for $1\leq i\leq k$ and hence $M=\omega(G_r)$. Therefore by Theorem \ref{3} the result follows. \end{proof}

\begin{cor}\label{4}
If $G=K_{n_1,n_2,\ldots, n_k}$ (where $n_i\geq 2$) is a complete $k$-partite graph, then $box(G)=k$.
\end{cor}
%\begin{thm}\label{3.1}
%Let $G$ be a graph with $V(G)=\{u_1,u_2,\ldots, u_n\}$ such that $\langle \{u_1, u_2,\ldots, u_{\ell}\}\rangle$ is a clique in $G$ and $N_G[u_i]=N_G[u_j]$, for $1\leq i\leq j\leq \ell$. If $H_1,H_2,\ldots, H_{\ell}$ are pairwise disjoint complete graphs and $H_{\ell+1},\ldots, H_n$ are pairwise disjoint graphs, then $box(G[H_1,H_2,\ldots, H_n])\leq \mathop\sum\limits_{i=\ell+1}^{n} box(H_i)+1$.
%\end{thm}
%\begin{proof}
%Using the similar arguments employed in the proof of Theorem \ref{3} after identify $u_1,u_2,\ldots, u_{\ell}$ by a single vertex $u$, we get the result.
%\end{proof}
\begin{thm}\label{3.1} 
Let $G$ be a graph with $V(G)=\{u_1,u_2,\ldots, u_n\}$ such that $\langle \{u_1, u_2,\ldots, u_{\ell}\}\rangle$ is a clique in $G$ and $\ell<n$. If $H_1,H_2,\ldots, H_{\ell}$ are pairwise disjoint complete graphs and $H_{\ell+1},\ldots, H_n$ are pairwise disjoint graphs, then $box(G[H_1,H_2,\ldots, H_n])\leq \mathop\sum\limits_{i=\ell+1}^{n} box(H_i)$. 
\end{thm}
\begin{proof}
For $\ell+1\leq i\leq n$, let $box(H_i)=k_i$. By Theorem \ref{1a}, there exist  interval graphs $I_i^1,\ldots, I_i^{k_i}$ with vertex set $V(H_i)$ such that $H_i=\mathop\bigcap\limits_{j=1}^{k_i} I_i^j$. For $\ell+1\leq i\leq n$ and $1\leq j\leq k_i$, let $J_i^j$ be a graph as defined in Theorem \ref{3}. Then we can prove that $J_i^j$ is an interval graph and $G[H_1,H_2,\ldots, H_n]$ is a spanning subgraph of $J_i^j$ (the proof is same as in Theorem \ref{3}). 
We have to prove that $E(G[H_1,\ldots,H_n])=\mathop\bigcap\limits_{i=\ell+1}^n\mathop\bigcap\limits_{j=1}^{k_i}E(J_i^j)$. 

%As $G[H_1,H_2,\ldots, H_n]$ is a spanning subgraph of $J_i^j$, for $\ell+1\leq i \leq n$ and $1\leq j\leq k_i$, 
It is enough to prove that, if $xy\notin E(G[H_1,H_2,\ldots, H_n])$, then there exist, $\ell+1\leq i\leq n$ and $1\leq r\leq k_i$ such that $xy\notin E(J_i^r)$. 

If $x\in V(H_i)$ and $y\in V(H_j)$ for $\ell+1\leq i\leq j\leq n$, then by the same proof as in Theorem \ref{3}, $xy\notin E(J_i^r)$, for some $\ell+1\leq i\leq n$ and $1\leq r\leq k_i$. 

The only other possibility is: %either $x$ or $y\in V(H_i)$, say 
$x\in V(H_i)$ for some $1\leq i\leq \ell$ and $y\in V(H_j)$, for some $\ell+1\leq j\leq n$. In this case, $u_i$ and $u_j$ are not adjacent in $G$ and hence not adjacent in $R_j$ (see Theorem \ref{3}) and hence by the definition of $J_j^r$, $xy\notin J_j^r$, for all $1\leq r\leq k_j$. 
%From Claims 1, 2 and by Theorem \ref{1a}, the proof follows.
\end{proof}
\section{Boxicity of Zero-Divisor Graph of a ring}
In this section, we obtain a bound for boxicity of the zero-divisor graph of a ring.

\noindent Let $R$ be a ring. By Theorem \ref{A}, $\Gamma(R)\cong \Gamma_E(R)[\left\langle C_{x_1}\right\rangle,\left\langle C_{x_2}\right\rangle,\ldots, \left\langle C_{x_k}\right\rangle]$, where $|V(\Gamma_E(R))|=k$. %Since for $x^2=0$, $\left\langle C_{x} \right\rangle$ is complete and by Theorem \ref{3.1}, we have

\begin{thm}\label{4.1}
%If $R$ is a ring and $\Gamma_E(R)$ is the compressed zero-divisor graph of $R$, then $box(\Gamma(R))\leq k$, where $k=|V(\Gamma_E(R))|$.
If the subgraph graph induced by $\{x_1,x_2,\ldots, x_r\}$ is a clique in $\Gamma_E(R)$ with $x^2=0$, for $1\leq i\leq r$, then $box(\Gamma(R))\leq k-r$, where $k=|V(\Gamma_E(R))|$.
\end{thm}
\begin{proof}
%If $x^2\neq 0$, then $\left\langle C_{x_i} \right\rangle$ is a complement of a complete graph, by Theorem \ref{A} and hence $box(\left\langle C_{x} \right\rangle)=1$, by Theorem \ref{3.1}.
By Theorem \ref{A},  for $1\leq i\leq k$, $\left\langle C_{x_i} \right\rangle$ is a complete graph or a totally disconnected graph. Then $box(\left\langle C_{x_i} \right\rangle)=1$, for $1\leq i\leq k$. By Theorem \ref{3.1}, $box(\Gamma(R))\leq k-r$.
%By Theorem \ref{A}, $\left\langle C_{x_i} \right\rangle$, for $1\leq i\leq r$ is a complete graph and $\left\langle C_{x_i} \right\rangle$ complement of a complete graph, for $r+1\leq i\leq k$. Then $box(\left\langle C_{x} \right\rangle)=1$. Hence $box(\Gamma(R))\leq k-r$, by Theorem \ref{3.1}.
\end{proof}
It is proved in \cite{And1} that if $R$ is a non-zero reduced ring, then there exists positive integer $k$ such that the compressed zero-divisor graph $\Gamma_E(R)$ of $R$ is isomorphic to $\Gamma(R')$, that is $\Gamma_E(R)\cong\Gamma(R')$, where $R'=\mathbb{Z}_2^k=\mathbb{Z}_2\times \mathbb{Z}_2 \times \ldots \times \mathbb{Z}_2$ ($k$-copies). It is proved that in \cite{Beck}, $\chi(\Gamma(R))=k=\omega(\Gamma(R))$.

\begin{thm}
If $R$ is a non-zero reduced ring, then $\chi(\Gamma(R))\leq box(\Gamma(R))\leq 2^{\chi(\Gamma(R))}-2$.
\end{thm}
\begin{proof} By Theorem \ref{A}, $\Gamma(R)\cong \Gamma(\mathbb{Z}_2^k)[C_{x_1},\ldots, C_{x_{{2^k-2}}}]$. Since $R$ is reduced, $x_i^2\neq 0$, for $1\leq i\leq 2^k-2$. Therefore, $\left\langle C_{x_i} \right\rangle$'s are complement of complete graphs. Hence, by Theorem \ref{3}, $box(\Gamma(R))\leq 2^k-2$ (because $box(\left\langle C_{x_i} \right\rangle)=1$ and $|V(\mathbb{Z}_2^k)|= 2^k-2$). %Since $|V(R')|= 2^k-2$, (by Theorem \ref{A})

Next we prove the lower bound. For $1\leq i\leq k$, let $e_i$ be the element of $R'$ with the $i^{th}$-coordinate is one and other coordinate are zero. Clearly $\{e_i\ |\ 1\leq i\leq k\}$ forms a clique in $\Gamma(R)$ and $e_i^2\neq 0$, for $1\leq i\leq k$. Thus, $\left\langle C_{e_i} \right\rangle$ is the complement of complete graph and hence by Theorem \ref{3}, $k\leq box(\Gamma(R))$, because $box(\left\langle C_{e_i} \right\rangle)=1$.
\end{proof}

\section{Boxicity of Zero-Divisor Graph of the ring $\mathbb{Z}_N$}%Application 2}%
In Section 5, we first prove that the chromatic number and the clique number of $\Gamma_E(\mathbb{Z}_N)$ are same for any positive integer $N$.

\begin{lem}\label{1.1}
If $N= \prod\limits_{i=1}^{a}p_i^{2n_i} \prod\limits_{j=1}^{b}q_j^{2m_j+1}$, where $p_i$'s and $q_j$'s are distinct  prime numbers,
then 
$\omega(\Gamma_E(\mathbb{Z}_N))=\chi(\Gamma_E(\mathbb{Z}_N))=\mathop\prod\limits_{i=1}^{a}(n_i+1)\mathop\prod\limits_{j=1}^{b}(m_j+1)+b-1.$
\end{lem}

\begin{proof}
Let $S=\{d\in V(\Gamma_E(\mathbb{Z}_N)) : d\mid N \mid d^2 \}$. Then, clearly $S$ is a clique in $\Gamma_E(\mathbb{Z}_N)$.\\
\textbf{Claim.}  $|S|= \prod\limits_{i=1}^{a}(n_i+1) \prod\limits_{j=1}^{b}(m_j+1)-1$. 
%\begin{proof}
Let $d$ be a positive integer such that 
\begin{equation} \label{div}
d\mid N\mid d^2
\end{equation}
Since $d\mid N$, we have  $d$ is of the form $\prod\limits_{i=1}^{a}p_i^{\alpha_i} \prod\limits_{j=1}^{b}q_j^{\beta_j} $, where $0\leq \alpha_i\leq 2n_i$, $\forall  \ i$ and $0\leq \beta_j\leq 2m_j+1$, $\forall \ j$. Since $\frac{N}{d}\mid  d$, we have $$2n_i-\alpha_i\leq \alpha_i\mbox{ and }2m_j+1-\beta_j\leq \beta_j$$ for all $1\leq i\leq a$ and  $1\leq j\leq b$. Thus for all $1\leq i\leq a$, $1\leq j\leq b$, we have $$n_i\leq \alpha_i\leq 2n_i \mbox{ and }m_j+1 \leq \beta_j\leq 2m_j+1.$$ 
Therefore, the number of positive integers $d$ satisfying the condition \eqref{div} is $\prod\limits_{i=1}^{a}(n_i+1) \prod\limits_{j=1}^{b}(m_j+1)$. 
As the set of all positive integers satisfying the condition \eqref{div} is $S\cup \{N\}$, we get that 
$$|S|= \prod\limits_{i=1}^{a}(n_i+1) \prod\limits_{j=1}^{b}(m_j+1)-1.$$ 
Hence the claim. 

%\end{proof}
%Let $\overline{d}=p_1^{n_1}\ldots p_a^{n_a}q_1^{m_1+1}\ldots q_b^{m_b+1}\in V(\Gamma_E(\mathbb{Z}_N))$. 
For  $1\leq \eta\leq b$, we let
\begin{equation} \label{xieta}
 \xi_\eta= \prod\limits_{i=1}^a p_i^{2n_i}  \prod\limits_{j=1\atop j\ne \eta }^bq_j^{2m_j+1} \times q_{\eta}^{m_{\eta}}.
\end{equation}
Then $\xi_\eta\notin S$ and $T=S\cup \{\xi_\eta : 1\leq \eta\leq b\}$ forms a clique in $\Gamma_E(\mathbb{Z}_N)$. Note that 
 \begin{equation} \label{cardT}
 |T|= \prod\limits_{i=1}^{a}(n_i+1)\mathop\prod\limits_{j=1}^{b}(m_j+1)+b-1.
 \end{equation}
Hence 
 $\omega(\Gamma_E(\mathbb{Z}_N))\geq  |T|$. \\
Next, we prove that $\chi(\Gamma_E(\mathbb{Z}_N))\leq |T|$. Let

%Let $c$ be a coloring on the clique $T$, which is an injective map.   
\begin{equation}\label{dfact}
d=\prod\limits_{i=1}^{a}p_i^{r_i} \prod\limits_{j=1}^{b}q_j^{s_j}\in V(\Gamma_E(\mathbb{Z}_N))\setminus T.
\end{equation} 
Then, there exists $i$ with $1\leq i\leq a$ such that $r_i<n_i$ or there exists $j$ with $1\leq j\leq b$ such that $s_j\leq m_j$. 

\noindent \textbf{Case 1.} $r_i<n_i$  for some $i$ with $1\leq i\leq a$.\\
If $\sigma(d)= \max \{i : r_i<n_i, 1\leq i\leq a, \mbox{ where }r_i \mbox{ is as in \eqref{dfact}}\}$ and  $\overline{d}=\prod\limits_{i=1\atop i\ne \sigma(d)}^a p_i^{2n_i}  \times p_{\sigma(d)}^{n_{\sigma(d)}} \times \prod\limits_{j=1}^bq_j^{2m_j+1}$, then $\overline{d}\in S\subset T$ and $N\nmid d\overline{d}$, as $r_{\sigma(d)}+n_{\sigma(d)} <2n_{\sigma(d)}$. So $d$ and $\overline{d}$ are not adjacent in $\Gamma_E(\mathbb{Z}_N)$. 
%Color $d$ by $c(d)=c(\overline{d})$.
% where $d'=p_1^{2n_1}\ldots p_{\ell-1}^{n_{\ell-1}}p_{\ell}^{n_{\ell}}p_{\ell+1}^{n_{\ell+1}}\ldots p_a^{2n_a}q_1^{2m_1+1}\ldots q_b^{2m_b+1}$

\noindent \textbf{Case 2.} $r_i\geq n_i$ for all $1\leq i\leq a$ and there exists $j$ with $1\leq j\leq b$ such that $s_j\leq m_j$.\\ If $\ell(d)=\max\{j : s_j\leq m_j, 1\leq j\leq b\mbox{ where }s_j \mbox{ is as in \eqref{dfact}}\}$, then $\xi_{\ell(d)}\in T\setminus S$ (see \eqref{xieta} for the definition of $\xi_\eta$) and $N\nmid \xi_{\ell(d)}d$, as $s_{\ell(d)} +m_{\ell(d)} < 2m_{\ell(d)} +1$. Therefore, 
 $d$ and $\xi_{\ell(d)}$ are not adjacent in $\Gamma_E(\mathbb{Z}_N)$. 

%Using \eqref{cardT}, we can choose a bijective function 
Let $\psi: T \to \left\{1,2,\ldots, |T|\right\}$ be a bijective function. 
We extend the function $\psi$ over $V(\Gamma_E(\mathbb{Z}_N))$ as  $c:V(\Gamma_E(\mathbb{Z}_N))\to \left\{1,2,\ldots, |T|\right\}$ defined by 
$$c(d) = \begin{cases} \psi(d) & \mbox{ if }d\in T\\
\psi(\overline{d}) & \mbox{ if }d\not\in T \mbox{ and } d \mbox{ as in {\bf Case 1}}\\
\psi(\xi_{\ell(d)}) & \mbox{ if } d\not\in T \mbox{ and } d \mbox{ as in {\bf Case 2}}
\end{cases}$$

%$d'=p_1^{2n_1}\ldots p_{\ell-1}^{2n_{\ell-1}}p_{\ell}^{n_{\ell}}p_{\ell+1}^{2n_{\ell+1}}\ldots p_a^{2n_a}q_1^{2m_1+1}\ldots q_b^{2m_b+1}$.
%Now color $c(d)=k_d$, where  
%\[
%k_d=
%\begin{cases}
%$min$ \{j  : r_j<n_j\},  & \text{if  there exists $i, 1\leq i\leq a,$ with $r_i<n_i$,}\\
%$min$ \{j  : s_j\leq m_j\}, & \text{otherwise}.
%\end{cases}
%\]
Finally, we prove that $c$ is a proper coloring of $\Gamma_E(\mathbb{Z}_N)$. Let $d',d''\in  V(\Gamma_E(\mathbb{Z}_N))$.  
\begin{itemize}
\item If $d', d''\in T$, then they receive distinct colors. 
\item If $d', d''\in V(\Gamma_E(\mathbb{Z}_N)) \backslash T$ and $d'$ is under Case 1 and $d''$ is under case 2, then  they receive different colors, as $\overline{d}\in S$ and $\xi_{\ell(d)}\in T\setminus S$.
\item If $d', d''\in V(\Gamma_E(\mathbb{Z}_N)) \backslash T$, both $d', d''$ are under Case 1, and $c(d')=c(d'')$, we show that $d'$ and $d''$ are not adjacent in $\Gamma_E(\mathbb{Z}_N)$. Indeed, $\psi(\overline{d'}) = \psi(\overline{d''})$ for some $\overline{d'}, \overline{d''}\in S$ corresponding to $d'$ and $d''$, respectively, as introduced in Case 1. Since $\psi$ is injective, we have  $\overline{d'} = \overline{d''}$ and hence $\sigma(d')=\sigma(d'')=\sigma$, say. This implies that $r_{\sigma}(d')+r_{\sigma}(d'')<2n_{\sigma}$, where $r_{\sigma}(d')$ and $r_{\sigma}(d'')$ are the powers of $p_\sigma$ in $d'$ and $d''$, respectively. %as in \eqref{dfact}. 
Thus, $N\nmid d'd''$, and hence $d'$ and $d''$ are not adjacent in $\Gamma_E(\mathbb{Z}_N)$.
\item If $d', d''\in V(\Gamma_E(\mathbb{Z}_N)) \backslash T$, both $d', d''$ are under Case 2, and $c(d')=c(d'')$, then $d'$ and $d''$ are not adjacent in $\Gamma_E(\mathbb{Z}_N)$, by a similar argument used in the previous case.
\item If $d'\in T$, $d''\in  V(\Gamma_E(\mathbb{Z}_N)) \backslash T$, and $c(d')=c(d'')$, then $d'=\overline{d''}$ or $d'=\xi_{\ell(d'')}$. If $d'=\overline{d''}$, then $d'$ and $d''$ are not adjacent, as $n_{\sigma(d'')}+r_{\sigma(d'')}(d'')<2n_{\sigma(d'')}$, equivalently, $N\nmid d'd''$. Similarly, we can prove that $d'$ and $d''$ are not adjacent whenever $d'=\xi_{\ell(d'')}$.
\end{itemize}

 Therefore, $\chi(\Gamma_E(\mathbb{Z}_N))\le |T|$. Thus, the proof is completed.
\end{proof}

As an application of Theorem \ref{4.1} and Lemma \ref{1.1}, we have % we prove the following result. and Lemma \ref{1.1}
\begin{thm}
If $N= \prod\limits_{i=1}^{a}p_i^{2n_i} \prod\limits_{j=1}^{b}q_j^{2m_j+1}$, where $p_i$'s and $q_j$'s are distinct  prime numbers. % and $a,b$ are positive integers. 
Then $box(\Gamma(\mathbb{Z}_N))\leq \left(\mathop\prod\limits_{i=1}^{a}(2n_i+1)\mathop\prod\limits_{j=1}^{b}(2m_j+2)\right)-\left(\mathop\prod\limits_{i=1}^{a}(n_i+1)\mathop\prod\limits_{j=1}^{b}(m_j+1)\right)-1.$
\end{thm}
\begin{proof} Let $\Gamma_E(\mathbb{Z}_N)$ be the compressed zero-divisor graph of $\mathbb{Z}_N$. Then the vertex set $V(\Gamma_E(\mathbb{Z}_N))$ of $\Gamma_E(\mathbb{Z}_N)$ is the set of all proper divisors, $d_1,d_2,\ldots, d_k$, of $N$. Hence 
 $|V(\Gamma_E(\mathbb{Z}_N))|=k=\mathop\prod\limits_{i=1}^a(2n_i+1)\mathop\prod\limits_{j=1}^b(2m_j+2)-2$.
 
 Let $S=\{d\in V(\Gamma_E(\mathbb{Z}_N)) : d\mid N \mid d^2 \}$. Then by the Lemma \ref{1.1},  
\begin{equation}\label{CardS}
|S|=\mathop\prod\limits_{i=1}^{a}(n_i+1)\mathop\prod\limits_{j=1}^{b}(m_j+1)-1, 
 \end{equation}
 $S$ is a clique in $\Gamma_E(\mathbb{Z}_N)$ and $x^2=0$ for $x\in S$. By Theorem \ref{4.1}, %we have the result.
%By observations 4 and 5, 
%$$\Gamma (\mathbb{Z}_N)=\Gamma_E(\mathbb{Z}_N)\big[\langle A_{d_1}\rangle, \langle A_{d_2}\rangle,\ldots, \langle A_{d_k}\rangle\big]$$ and $\langle A_d\rangle$ is complete graph if $d\in S$, otherwise it is independent. In both cases, $box(\langle A_d\rangle)=1$. Hence  by Theorem \ref{3.1}, 
$box(\Gamma(\mathbb{Z}_N))\leq |V(\Gamma_R(R))|-|S|$.
%\mathop\sum\limits_{d\in V(\Gamma_E(\mathbb{Z}_N))\backslash S}box(\langle A_d\rangle)$. 
%By observation 3 and \eqref{CardS},  
Hence the result follows.
\end{proof}

Finally, we characterize a composite number $N$ for which the boxicity of the zero-divisor graph $\Gamma(\mathbb{Z}_N)$ is one. We can easily observe the following: If a graph $G$ contains an induced cycle of length 4 or more, then $box(G)\geq 2$, that  is, $G$ is not an interval graph. 

\begin{thm}
Let $N$ be any composite number. Then,  $box(\Gamma(\mathbb{Z}_N))=1$ if and only if either $N=p^n$ for some prime number $p$ and for some integer $n\ge 2$ or $N=2p$ for some odd prime number $p$.
\end{thm}
\begin{proof} 
For a proper divisor $d$ of $N$, define $A_d=\{a\in \mathbb{Z}_N: gcd(N,a)=d)\}$. It is clear that $A_{d_i}\cap A_{d_j}=\emptyset$ whenever $d_i$ and $d_j$ are distinct proper divisors of $N$ and $\Gamma(\mathbb{Z}_N)\cong\Gamma_E(\mathbb{Z}_N)[A_{d_1},A_{d_2},\ldots, A_{d_k}]$, where $d_i$'s are all proper divisors of $N$. 
First, let us assume that $box(\Gamma(\mathbb{Z}_N))=1$.
  
Suppose  $N=p_1^{n_1}p_2^{n_2}\ldots p_s^{n_s}$, where 
$p_i$'s are distinct primes. 

\noindent\textbf{Case 1.}  $n_i\geq 2$ for some $i$ with $1\leq i\leq s$. 

We  assume that $n_1\geq 2$. If $s\geq 2$, then the set of vertices $$\{(p_1^{n_1-1}p_2^{n_2}\ldots p_s^{n_s}),
(p_1^{n_1}p_3^{n_3}\ldots p_s^{n_s}),
(p_2^{n_2}\ldots p_s^{n_s}), p_1^{n_1}\}$$ in $\Gamma(\mathbb{Z}_N)$ induces a cycle of length 4, a contradiction. Therefore, $s=1$ and hence $N=p_1^{n_1}$.

\noindent\textbf{Case 2.} $n_i=1$, for all $1\leq i\leq s$.%$N=p_1p_2\ldots p_s$.

Since $N$ is composite number, we have $s\geq 2$. 
Suppose that there are two odd prime factors for $N$, say $p_1\geq 3$ and $p_2\geq 3$, then $K_2[\langle A_{p_1p_3p_4\ldots p_s}\rangle, \langle A_{p_2p_3\ldots p_s}\rangle]=\langle A_{p_1p_3p_4\ldots p_s}\rangle\vee \langle A_{p_2p_3\ldots p_s}\rangle$ is an induced subgraph of  $\Gamma(\mathbb{Z}_N)$. Since $A_{p_1p_3p_4\ldots p_s}$ and  $A_{p_2p_3\ldots p_s}$ are independent subsets of $\Gamma(\mathbb{Z}_N)$, the set $\{(p_1p_3p_4\ldots p_s),2(p_1p_3p_4\ldots p_s), (p_2p_3\ldots p_s),$\\ $ 2(p_2p_3\ldots p_s)\}$ of vertices of $\Gamma(\mathbb{Z}_N)$ induces a cycle of length $4$, (because $|A_{p_1p_3p_4\ldots p_s}|\geq 2$ and $|A_{p_2p_3\ldots p_s}|\geq 2$), a contradiction. So, $N=2p$ for some odd prime $p$. %Hence the first part follows.

Conversely, let us assume that either $N=p^n$ for some prime number $p$ and for some integer $n\ge 2$ or $N=2p$ for some odd prime number $p$. 
If $N=2p$, where $p$ is a prime, then $\Gamma(\mathbb{Z}_N)\cong K_2$ and hence $box(\Gamma(\mathbb{Z}_N))=1$. 
Now let us assume that  $N=p^n$, where $p$ is a prime and $n \geq 2$. When $n=2$, $\Gamma(\mathbb{Z}_{p^2})\cong K_{p-1}$ and hence $box(\Gamma(\mathbb{Z}_{p^2}))=1$. 
When $n=3$, $\Gamma(\mathbb{Z}_p^3)\cong \Gamma_E(\mathbb{Z}_{p^3})[\left\langle A_p\right\rangle,\left\langle A_{p^2}\right\rangle]$, where $A_p$ is an independent set, $\left\langle A_{p^2}\right\rangle$ is complete and every vertex of $A_p$ is adjacent to every vertex of $A_{p^2}$. Let $A_p=\{v_j:1\leq j\leq |A_p|\}$. We define an interval representation $f$ of $\Gamma(\mathbb{Z}_{p^3})$. For $x\in A_{p^2}$, define $f(x)=[0,1]$ and for $v_j\in A_p$, define $f(v_j)=\{\frac{1}{j}\}$. One can easily verify that $f$ is an interval representation of $\Gamma(\mathbb{Z}_{p^3})$ and therefore $box(\Gamma(\mathbb{Z}_N))=1$.

When $n\geq 4$, $\Gamma(\mathbb{Z}_N)=\Gamma_E(\mathbb{Z}_N)[\left\langle A_p\right\rangle,\left\langle A_{p^2}\right\rangle,\ldots, \left\langle A_{p^{n-1}}\right\rangle]$. For $1\leq i\leq n-1$, let $A_{p^i}=\{v_{i,j}\ : 1\leq j\leq |A_{p^i}|\}$. Here one can easily observe the following.
\begin{itemize}
%\item %For $i\in \{1,2,\ldots, \left\lfloor \frac{n}{2}\right\rfloor\}$, every vertex of $A_{p^i}$ is adjacent to every vertex of $A_{p^{n-j}}$ if $j\in \{1,2,\ldots, i\}$; and 
\item If $n$ is even, then we see that the following.\\
(i) $A_{p^i}$ is an independent set if and only if $i\in \{1,2,\ldots, \frac{n}{2}-1\}$; \\
(ii) $\left\langle A_{p^j}\right\rangle$ is complete if and only if $j\in \{\frac{n}{2}, \ldots, n-1\}$;\\
(iii) For $i\in \{1,2,\ldots, \frac{n}{2}-1\}$, every vertex of $A_{p^i}$ is adjacent to every vertex of $A_{p^{n-j}}$ if and only if  $j\in \{1,2,\ldots, i\}$;\\
(iv) the set $\mathop\bigcup\limits_{j=1}^{\frac{n}{2}-1}A_{p^j}$ is independent and  $\left\langle \mathop\bigcup\limits_{j=\frac{n}{2}}^{n-1}A_{p^j} \right\rangle$ is complete.
\item If $n$ is odd, then we see that the following.\\
(i) $A_{p^i}$ is an independent set if and only if $i\in \{1,2,\ldots, \left\lfloor \frac{n}{2} \right\rfloor\}$;\\
(ii) $\left\langle A_{p^j}\right\rangle$ is complete if and only if $j\in \{\left\lceil \frac{n}{2}\right\rceil, \ldots, n-1\}$;\\
(iii) For $i\in \{1,2,\ldots, \left\lfloor \frac{n}{2}\right\rfloor\}$, every vertex of $A_{p^i}$ is adjacent to every vertex of $A_{p^{n-j}}$ if and only if $j\in \{1,2,\ldots, i\}$;\\
(iv) the set $\mathop\bigcup\limits_{j=1}^{\left\lfloor \frac{n}{2}\right\rfloor}A_{p^j}$ is an independent set and $\left\langle \mathop\bigcup\limits_{j=\left\lceil \frac{n}{2}\right\rceil}^{n-1}A_{p^j} \right\rangle$ is complete.
\end{itemize}
%\begin{center}
%\begin{figure}
%\includegraphics[width=5in]{Figure-1}
%\caption{Interval representation of $\Gamma(\mathbb {Z}_{p^n})$, when $n$ is even}
%\end{figure}
%\end{center}
%
%\begin{center}
%\begin{figure}
%\includegraphics[width=5in]{Figure-2}
%\caption{Interval representation of $\Gamma(\mathbb{Z}_{p^n})$, where $n$ is odd}
%\end{figure}
%\end{center}
We now construct an interval representation $f$ of $\Gamma(\mathbb{Z}_N)$ as follows. %, see Figure 1 and 2.\\

For $1\leq i\leq \left\lfloor \frac{n}{2}\right\rfloor-1$, define
$$f(v_{n-i,j})= \left[0,\left\lfloor \frac{n}{2} \right\rfloor-i+1\right], \mbox{ for every } j\in\{1,2,\ldots, |A_{p^{n-i}}|\}.$$

 For $\left\lceil \frac{n}{2}\right\rceil+1\leq i\leq n-1$, %if $A_{p^{n-i}}=\{v_{n-i,j} : 1\leq j\leq |A_{p^{n-i}}|\},$ and 
let $x_{n-i,1},x_{n-i,2},\ldots, x_{n-i,|A_{p^{n-i}}|}$ be distinct points in %in the open interval
$\left(i-\left\lceil \frac{n}{2}\right\rceil, i-\left\lceil \frac{n}{2}\right\rceil+1\right)$. Define
$$f(v_{n-i,j})=\{x_{n-i,j}\}, \mbox{ for every } j\in \{1,2,\ldots, |A_{p^{n-i}}|\}.$$
%\mbox{a distinct singleton set in} \left(i-\left\lceil \frac{n}{2}\right\rceil, i-\left\lceil \frac{n}{2}\right\rceil+1\right).$$ 

\noindent If $n$ is even, define $$f\big(v_{\frac{n}{2},j}\big)=[0,1], \mbox{ for every } j\in \{1,2,\ldots,  |A_{p^{\frac{n}{2}}}|\}.$$

\noindent If $n$ is odd, consider distinct points $y_{\left\lfloor \frac{n}{2}\right\rfloor,1},y_{\left\lfloor \frac{n}{2}\right\rfloor,2},\ldots, y_{\left\lfloor \frac{n}{2}\right\rfloor,|A_{p^{\left\lfloor \frac{n}{2}\right\rfloor}}|}$ in $(0,1)$ and define 
$$f\Big(v_{\left\lfloor \frac{n}{2}\right\rfloor,j}\Big)=\Big\{y_{\left\lfloor \frac{n}{2}\right\rfloor,j}\Big\}, \mbox{ for every } j\in\Big\{1,2,\ldots, |A_{p^{\left\lfloor \frac{n}{2}\right\rfloor}}|\Big\}$$ and 
$$f\Big(v_{\left\lceil \frac{n}{2}\right\rceil,j}\Big)=[0,1], \mbox { for every } j\in\{1,2,\ldots, |A_{p^{\left\lceil \frac{n}{2}\right\rceil}}|\}.$$
%$A_{p^{\left\lfloor \frac{n}{2}\right\rfloor}}=\{v_{\left\lfloor \frac{n}{2}\right\rfloor,j} : 1\leq j\leq |A_{p^{\left\lfloor \frac{n}{2}\right\rfloor}}|\}$ and $y_1,y_2,\ldots, y_{|A_{p^{\left\lfloor \frac{n}{2}\right\rfloor}}|}$ are distinct points in $(0,1)$, define 
%\noindent If $A_{p^{\left\lfloor \frac{n}{2}\right\rfloor}}=\{v_j : 1\leq j\leq |A_{p^{n-\left\lceil \frac{n}{2}\right\rceil}}|\},$ and $y_1,y_2,\ldots, y_{|A_{p^{n-i}}|}$ are distinct elements in $(0,1)$, define
%$$f(v) = \begin{cases}\{y_j\},    & \mbox{if } v=v_{\left\lfloor \frac{n}{2}\right\rfloor,j},\\
%[0,1], & \mbox{if } v\in A_{\left\lceil \frac{n}{2}\right\rceil}.
%\end{cases}$$
One can verify that $f$ is an interval representation of $\Gamma(\mathbb{Z}_N)$ and hence $box(\Gamma(\mathbb{Z}_N))=1$.

%The proof for the cases $N=p^2 $ and $p^3$ are clear, as $\Gamma(\mathbb{Z}_{p^2})$ is $K_{p-1}$ and $\Gamma(\mathbb{Z}_{p^3}) =\Gamma_E(\mathbb{Z}_N)[\langle A_p\rangle,\langle A_{p^2}\rangle]$.
\end{proof}
\subsection*{Acknowledgment}  This research was supported by the University Grant Commissions Start-Up Grant, Government of India grant No. F. 30-464/2019 (BSR) dated 27.03.2019.
%The author thank Xuding Zhu for helpful discussion. We also thank the two referees for their valuable comments.
%This research was supported by the Department of Science and
%Technology, Government of India grant DST SR / S4 / MS: 234 / 04
%dated March 31, 2006.

\end{document}